\documentclass[english, reqno]{amsart}
\usepackage{amsfonts, amsmath, amssymb, amscd, amsthm, array, graphicx}
\usepackage[latin9]{inputenc}
\usepackage{amsmath,comment}
\usepackage{stmaryrd}
\usepackage{faktor} 
\usepackage{graphicx}
\usepackage{amssymb}
\usepackage{mathrsfs}
\usepackage{dsfont}
\usepackage{babel}
\usepackage{color}
\usepackage[all]{xy}
\usepackage[colorlinks=true, linkcolor=blue, citecolor=blue, urlcolor=blue, breaklinks=true]{hyperref}
\usepackage[capitalise]{cleveref}

\usepackage{mathtools}
\usepackage{hyperref,url}

\def\blfootnote{\xdef\@thefnmark{}\@footnotetext}
\makeatother
\usepackage{tikz}
\usetikzlibrary{arrows,shapes,snakes,automata,backgrounds,petri}
\usetikzlibrary{automata} 
\usetikzlibrary[automata] 
\usetikzlibrary{matrix,decorations.pathreplacing, decorations.markings}
\usepackage{geometry}
\usepackage{pdflscape}
\usepackage{graphicx}
\usetikzlibrary{external,automata,trees,positioning,shadows,arrows,shapes.geometric}

\usepackage{tikz-cd}
\usetikzlibrary{decorations.pathreplacing,decorations.markings}
\tikzset{close/.style={near start,outer sep=-2pt}} 
\tikzset{
  on each segment/.style={
    decorate,
    decoration={
      show path construction,
      moveto code={},
      lineto code={
        \path [#1]
        (\tikzinputsegmentfirst) -- (\tikzinputsegmentlast);
      },
      curveto code={
        \path [#1] (\tikzinputsegmentfirst)
        .. controls
        (\tikzinputsegmentsupporta) and (\tikzinputsegmentsupportb)
        ..
        (\tikzinputsegmentlast);
      },
      closepath code={
        \path [#1]
        (\tikzinputsegmentfirst) -- (\tikzinputsegmentlast);
      },
    },
  },
  mid arrow/.style={postaction={decorate,decoration={
        markings,
        mark=at position .5 with {\arrow[#1]{stealth}}
      }}},
}

\newlength\Textht
\newtheorem{thm}{Theorem}[section]
\newtheorem*{thm*}{Theorem}
\newtheorem{cor}[thm]{Corollary}
\newtheorem{lem}[thm]{Lemma}
\newtheorem{prop}[thm]{Proposition}

\theoremstyle{definition}
\newtheorem{defn}[thm]{Definition}

\newtheorem{obs}[thm]{Observation}

\newtheorem{rem}[thm]{Remark}
\newtheorem{exam}[thm]{Example}

\newcommand{\M}[1]{{\color{magenta} #1}}

\newcommand{\ZZ}{\mathbb{Z}}
\newcommand{\gh}{\widetilde{H_{\Gamma}}}
\newcommand{\im}{\texttt{Im}}

\begin{document}

\title{Presentation of kernels of rational characters of right-angled Artin groups}

\author{Montserrat Casals-Ruiz}
\address{Ikerbasque - Basque Foundation for Science and Matematika Saila,UPV/EHU,  Sarriena s/n, 48940, Leioa - Bizkaia, Spain.} 
\email{montsecasals@gmail.com}

\author{Ilya Kazachkov}
\address{Ikerbasque - Basque Foundation for Science and Matematika Saila,UPV/EHU,  Sarriena s/n, 48940, Leioa - Bizkaia, Spain.} 
\email{ilya.kazachkov@gmail.com}

\author{Mallika Roy}
\address{Department of Mathematics, Universidad del Pa\`is Vasco (UPV/EHU), Spain.} \email{mallika.roy@upc.edu}

\subjclass{20E05, 20E36, 20K15}

\keywords{right-angled Artin group, Bestvina-Brady group, finiteness property.}

\begin{abstract}
In this note, we characterise when the kernel of a rational character of a right-anlged Artin group, also known as generalised Bestiva-Brady group, is finitely generated and finitely presented. In these cases, we exhibit a finite generating set and a presentation. These results generalise Dicks and Leary's presentations of Bestina-Brady kernels and provide an algebraic proof for the results proven by Meier, Meinert, and VanWyk.
\end{abstract}

\maketitle

\section{Introduction}\label{intro}


Let $\triangle_\Gamma$ be a non-empty finite flag complex, that is a finite simplicial complex that contains a simplex bounding every complete subgraph of its $1$-skeleton. The associated right-angled Artin
group $G_\Gamma$ is the group given by the presentation with generators the vertex set $V(\triangle_\Gamma)=\{a_1, \dots, a_s\}$, and
relators the commutators $[a_i,a_j]$ for each edge $(a_i,a_j)$ in $\triangle_\Gamma$.

In their influential work \cite{BB}, Bestvina and Brady described the homological finiteness properties of the kernels $H_\Gamma$ of the epimorphisms from $G_\Gamma$ to $\mathbb Z$ that map every standard generator of $G_\Gamma$ to $1$ in terms of the topology of the flag complex $\triangle_\Gamma$. More precisely, they prove that $H_\Gamma$ is finitely presented if and only if $\triangle_\Gamma$ is simply connected and $H_\Gamma$ is of type $FP_{n+1}$ if and only if $\triangle_\Gamma$ is $n$-acyclic.

In \cite{DL}, Dicks and Leary gave an explicit presentation for $H_\Gamma$. The generators in the presentation correspond to the edges $e_i$ of $\triangle_\Gamma$ and in the case when $\triangle_\Gamma$ is simply connected, it is shown that the relations are of the form $e_1^{\epsilon}e_2^{\epsilon}e_3^{\epsilon}$ for each directed $3$-cycles $(e_1, e_2, e_3)$ of $\triangle_\Gamma$ and $\epsilon = \pm 1$. This result gave an independent and purely algebraic proof that $H_\Gamma$ is finitely presented when $\triangle_\Gamma$ is simply connected.

The work of Bestvina and Brady was later extended by Meier, Meinert, and VanWyk. In \cite{MMV}, the authors describe the homotopical and homological $\Sigma$-invariants of $G_\Gamma$ in terms of the topology of subcomplexes of $\triangle_\Gamma$. In particular, their work determines the finiteness properties of kernels of rational characters of $G_\Gamma$.

In this note, we generalise the work of Dicks and Leary and give an algebraic characterisation of when a rational character has finitely generated and finitely presented kernel and in this case, we exhibit a presentation. More precisely, let $\varphi:G_\Gamma \to \ZZ$ be an arbitrary epimorphism. We consider a transverse set $\{ w_m \, \, | \, \, m \in \ZZ\}$ for the kernel $\gh \leqslant G_\Gamma$, where $w_m$ is defined as $a_{1}^{mr_1} \cdots a_{\ell}^{mr_\ell}$ of $G_\Gamma$, where $m\in \ZZ$, $m\ne 0$ and $\varphi(w_m)= m(r_1n_1+\dots+r_\ell n_\ell)=m$ and $w_0=1$.

We consider the induced subcomplex $\triangle_{\Gamma_L}<\triangle_\Gamma$ defined by the set of vertices of $\triangle_\Gamma$ with non zero image, i.e. $V(\triangle_{\Gamma_L})=\{v\in V(\triangle_\Gamma)\mid \varphi(v)\ne 0\}$. We say that $\triangle_\Gamma$ is $0$-acyclic-dominating in $\triangle_\Gamma$, if each vertex in $\triangle_\Gamma$ is either in $\triangle_{\Gamma_L}$ or adjacent to a vertex in $\triangle_{\Gamma_L}$; and it is $1$-acyclic-dominating if for each edge $(a_i, a_j)$ in $\triangle_\Gamma$, either $a_i$ is a vertex in $\triangle_{\Gamma_L}$, or $a_j\in V(\triangle_\Gamma)$, or there exists $a_k\in V(\triangle_{\Gamma_L})$ such that $(a_i, a_k), (a_j,a_k)$ are edges in $\triangle_\Gamma$, see Definition \ref{defn:domination}.

Our main results are the following:

\begin{thm*}[Finite generation, see Theorem \ref{thm:finite_generation}] Let $G_\Gamma$ be a RAAG and let $\varphi:G_\Gamma \to \ZZ$ be a rational character. Then, the kernel $\ker \varphi=\gh$ is finitely generated if and only if $\Gamma$ is connected and $0$-acyclic-dominating. More precisely,
$$
\gh = \langle w_m a_i w_{m + n_i}^{- 1} \mid 0 \leqslant |m| < \ell rN^2 \rangle,
$$
where $N, n_i, r, \ell$ are constants determined by $\varphi$.
\end{thm*}

\begin{thm*}[Finite presentation, see Theorem \ref{thm:finite_presentation}]
Let $\triangle_{\Gamma_L}$ be simply connected and $1$-acyclic-dominating. Then, the kernel $\ker \varphi=\gh$ has the following finite presentation 
    $$
    \langle x_{m,i} \, \, \mid \, \, R_1', R_2, R_3\rangle, 
    $$
    where $x_{k,i}=w_k a_i w_{k+n_i}^{-1}$, $0 \leqslant |m| < \ell rN^2$, and the relations $R_1', R_2, R_3$ are defined as follows:  
\begin{itemize}
\item[$R_1'$:] for any directed $3$-cycle $(a_{i_1}, a_{i_2}, a_{i_3})$,
\[
 \tau(w_s
 e_{i_1,i_2}^{Nq} e_{i_2,i_3}^{Nq} e_{i_{3},i_1}^{Nq}  w_s^{-1}), \hspace{0.2cm} 0 \leqslant |s| < N, q=\pm 1;
\]
\item[$R_2$:] $\tau(w_s [a_i, a_j] w_s^{-1})=1, \hspace{0.2cm} 0 \leqslant |s| < N$, where $(a_i, a_j) \in E(\Gamma)$.
\item [$R_3$:] $x_{t,i}=\tau(w_t a_i (w_{t+n_i})^{-1})$ for $0\leqslant |t| < \ell^2r^2N^3$.
\end{itemize}
The map $\tau$ is a Reidemeister rewriting process, see Definition \ref{defn:Reidemeister_rewriting}, and $N, n_i, r, \ell$ are constants determined by $\varphi$.
\end{thm*}

In particular, we recover the following presentation given by Dicks and Leary for Bestvina-Brady groups:

\begin{cor}[cf. Proposition 2 and Corollary 3 of \cite{DL}]
Let $\triangle_\Gamma$ be simply connected and let $H_\Gamma$ be the Bestvina-Brady subgroup of $G_\Gamma$. Then $H_\Gamma$ admits the finite presentation
$$
    \langle x_{0,i} \, \, \mid \, \, \tau(e_{i_1,i_2}^{q} e_{i_2,i_3}^{q} e_{i_3,i_1}^{q}) \rangle, 
    $$
    where $x_{0,i}= a_i a_1^{- 1}$, $a_i\in V(\Gamma)$, $e_{i_r,i_s}=a_{i_r}a_{i_s}^{-1}$ and $e_{i_1,i_2} e_{i_2,i_3} e_{i_3,i_1}$ a directed $3$-cycle, $q=\{\pm 1\}$.
\end{cor}

\section{Preliminaries and Notation}\label{prim-nt}
\subsection{Graph theory} We recall basic definitions and fix basic notation from graph theory. 
Throughout this paper, we consider only finite simplicial graphs, i.e. finite graphs that have no loops and multi-edges.  Given a graph $\Gamma$, we denote the set of its vertices and edges by $V(\Gamma)$ and $E(\Gamma)$, respectively. We denote by $e=(a_i, a_j)$ an edge connecting vertices $a_i$ and $a_j$ and we say that the vertices $a_i$ and $a_j$ are \emph{adjacent}. Sometimes, we consider oriented edges and in this case, we denote by $\iota(e)$ and $\tau(e)$, the initial and the terminal vertices of the edge $e$, i.e. if $e=(a_i, a_j)$ is an oriented edge, then $e=(\iota(e), \tau(e))$.

Given any subset $V'$ of $V(\Gamma)$, the \textit{induced subgraph} (or full subgraph) on $V'$ is a graph $\Gamma'$ whose vertex set is $V'$, and two vertices are adjacent in $\Gamma'$ if and only if they are adjacent in $\Gamma$. 

The \textit{flag complex} defined by $\Gamma$ and denoted by $\triangle_\Gamma$, is a finite simplicial complex whose $0$ and $1$ skeleton is $\Gamma$ and $\triangle_\Gamma$ contains an ($n$)-simplex bounding each complete induced subgraph (with $n$-vertices) of its 1-skeleton $\Gamma$.


\subsection{Bestvina--Brady groups}\label{bb_intro}

\begin{defn}
Let $\Gamma$ be a finite simplicial graph with the vertex set $V(\Gamma)$ and the edge set $E(\Gamma)$. 
The right-angled Artin group $G_{\Gamma}$ associated to $\Gamma$ has the following finite presentation:
\[
G_\Gamma = \bigl\langle V(\Gamma) \mid [a_i,a_j]=1 \text{ for each edge } (a_i,a_j) \in E(\Gamma) \bigr\rangle.
\]
Let $\varphi \colon G_\Gamma \rightarrow \ZZ$ be the group homomorphism sending all generators of $G_\Gamma$ to the generator $1$ of $\ZZ$. We call this epimorphism \emph{the canonical epimorphism} (from $G_\Gamma$ to $\ZZ$). The \textit{Bestvina-Brady group} $H_\Gamma$ associated with $\Gamma$ is the kernel $\ker \varphi$ of the canonical epimorphism $\varphi$. 
\end{defn}

As we mentioned in~\Cref{intro}, the Bestvina-Brady groups were first introduced in the seminal work of Bestvina and Brady~\cite{BB} as an answer to a long-standing open question regarding the existence of non-finitely presented groups of type \textbf{FP}---a result based on homological group theory. More precisely, the authors proved the following 

\begin{thm}[\cite{BB}, Main Theorem]
Let $\triangle_\Gamma$ be a finite simplicial flag complex. Then
\begin{itemize}
    \item[(1)] $H_\Gamma$ is finitely generated if and only if $\triangle_\Gamma$ is connected.
    \item[(2)] $H_\Gamma$ is finitely presented if and only if $\triangle_\Gamma$ is simply-connected.
    \item [(3)] $H_\Gamma$ is of type $\textbf{FP}_{n+1}$ if and only if $\triangle_\Gamma$ is $n$-acyclic.
\end{itemize}
\end{thm}

This result generalised J. Stallings' example of a finitely presented but not of type $\textbf{FP}_3$ given as the kernel $H_\Gamma$ for the canonical epimorphism from $F_2 \times F_2 \times F_2$ to $\ZZ$, see \cite{S}; and R. Bieri's group of type $\textbf{FP}_n$ but not of type $\textbf{FP}_{n+1}$, which is the kernel $H_\Gamma$ of the canonical epimorphism from $\prod_{i=1, \dots, n+1} F_2^{(i)}$ to $\ZZ$, see \cite{B}.

In \cite{DL}, Dicks--Leary, gave an algebraic proof of the fact that if $\triangle_\Gamma$ is simply-connected, then the Bestvina-Brady group is finitely presented and exhibited an explicit presentation. Their results can be summarised in the following:

\begin{thm}[\cite{DL}, Theorem 1]
Let $\triangle_\Gamma$ be connected. The group $H_\Gamma$ has a presentation with generators the set of directed edges of $\Gamma$, and relators all words of the form $e^n_1 e^n_2 \cdots e^n_{\ell}$, where $\ell,n \in \ZZ, n \geqslant 0,\ell \geqslant 2$, and $(e_1, \ldots,e_\ell)$ is a directed cycle in $\Gamma$.  
 
Furthermore, if the flag complex $\triangle_\Gamma$ is simply connected. Then $H_\Gamma$ has the following finite presentation
\begin{equation*}
H_\Gamma = \langle e \in E(\Gamma) \mid ef = fe, ef = g \text{ if }\triangle(e, f, g) \text{ is a directed triangle } \rangle,
\end{equation*}
where the inclusion $f \colon H_\Gamma \hookrightarrow A_\Gamma$ is given by $f({e}) = a_ia^{-1}_j$ for every edge ${e} = (a_i, a_j)$ of $\Gamma$.
\end{thm}

\begin{figure}[h]
    \centering
    \begin{tikzpicture}
    \tikzset{
    edge/.style={draw=black,postaction={on each segment={mid arrow=black}}}
} 
\node[fill=black!100, state, scale=0.10, vrtx/.style args = {#1/#2}{label=#1:#2}] (A) [vrtx=below/$a_i$]     at (0, 0) {};
\node[fill=black!100, state, scale=0.10, vrtx/.style args = {#1/#2}{label=#1:#2}] (C) [vrtx=above/$a_j$]    at (1, 1) {};
\node[fill=black!100, state, scale=0.10, vrtx/.style args = {#1/#2}{label=#1:#2}] (B) [vrtx=below/$a_k$]     at (2.5, 0) {};

\draw[edge] (A) -- (B) node[midway, below] {$g$};
\draw[edge] (C) -- (B) node[midway, above right] {$f$};
\draw[edge] (A) -- (C) node[midway, above left] {$e$};
    \end{tikzpicture}
    \caption{A directed triangle. 
    }
    \label{directed triangle}
\end{figure}
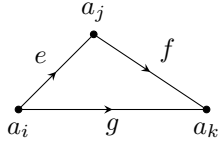

The Dicks-Leary presentation is not necessarily optimal, i.e. there are some redundant generators. A simpler presentation was given by Papadima-Suciu in~\cite{PS}, where the authors proved that it is sufficient to consider the edges of a spanning tree of $\Gamma$ as the generators of $H_\Gamma$.

\begin{exam}
Let $\Gamma$ be the graph in \Cref{fig:exam}. 
Choosing the spanning tree $T = \{e_1 , \ldots , e_5\}$ as indicated, the presentation of the Bestvina-Brady group is given as follows:
\[
H_\Gamma = \langle e_1,  \ldots, e_5 \mid [e_1 , e_2], [e_2 , e_3], [e_3 , e_4], e_5{e_2}^{-1}e_3 = {e_2}^{-1}e_3 e_5 \rangle.
\]
This particular $H_\Gamma$ is not isomorphic to any right-angled Artin group (see \cite[Proposition 9.4]{PS} for details). In \cite{CR}, Chang--Ruffoni characterized when the group $H_\Gamma$ associated to a $2$-dimensional $\triangle_\Gamma$ is a right-angled Artin group (see \cite[Theorem A]{CR}).
\begin{figure}[h]
    \centering
\begin{tikzpicture}[shorten >=1pt,node distance=20cm,auto]
\tikzset{
    edge/.style={draw=black,postaction={on each segment={mid arrow=black}}}
}
\node[fill=black!100, state, scale=0.10, vrtx/.style args = {#1/#2}{label=#1:#2}] (1) [vrtx=left/$x_6$] {};

\node[fill=black!100, state, scale=0.10, vrtx/.style args = {#1/#2}{label=#1:#2}] (2) [vrtx=right/$x_4$] [ below right of = 1] {};

\node[fill=black!100, state, scale=0.10, vrtx/.style args = {#1/#2}{label=#1:#2}] (3) [vrtx=left/$x_5$] [ below left of = 1] {};

\node[fill=black!100, state, scale=0.10, vrtx/.style args = {#1/#2}{label=#1:#2}] (4) [vrtx=right/$x_3$] [ below right of = 2] {};

\node[fill=black!100, state, scale=0.10, vrtx/.style args = {#1/#2}{label=#1:#2}] (5) [vrtx=below/$x_2$] [ below left of = 2] {};

\node[fill=black!100, state, scale=0.10, vrtx/.style args = {#1/#2}{label=#1:#2}] (6) [vrtx=left/$x_1$] [ below left of = 3] {};


\draw[edge] (6) -- (5) node[midway, below] {$e_1$};
\draw[edge] (5) -- (3) node[midway, left] {$e_2$};
\draw[edge] (5) -- (2) node[midway, right] {$e_3$};
\draw[edge] (5) -- (4) node[midway, below] {$e_4$};
\draw[edge] (2) -- (1) node[midway, right] {$e_5$};
\draw[edge] (6) -- (3);
\draw[edge] (3) -- (2);
\draw[edge] (3) -- (1);
\draw[edge] (4) -- (2);
\end{tikzpicture}
\caption{}
\label{fig:exam}
\end{figure}
\end{exam}

\section{Main results}
In \cite{MMV} Meier--Meinert--VanWyk described the homotopical and homological $\Sigma$-invariants of $G_\Gamma$ in terms of the topology of subcomplexes of $\triangle_\Gamma$. In particular, their work determines the finiteness properties of kernels of rational characters of $G_\Gamma$. They introduce the following

\begin{defn}\cite[Definition 1]{MMV}\label{defn:domination}
A subcomplex $L$ of a simplicial complex $K$ is $(-1)$-\emph{acyclic-dominating} if it is non-empty, or equivalently, $(-1)$-acyclic. 

For $n\geqslant 0$, $L$ is an $n$-\emph{acyclic-dominating} subcomplex of $K$, if for any vertex $v\in K\setminus L$, the ``restricted link" $lk_L(v) = lk(v) \cap L$ is $(n-1)$-acyclic and an $(n-1)$-acyclic-dominating subcomplex of the ``entire link" $lk(v)$ of $v$ in $K$. 
\end{defn}

Recall  that $\triangle_{\Gamma_L}$ is the induced subcomplex of $\triangle_\Gamma$ spanned by the vertices that are mapped nontrivially by the rational character.

With these notions, the authors proved the following

\begin{thm}\cite[Corollary A]{MMV}
Let $\Gamma$ be a simplicial graph, let $\triangle_\Gamma$ be the induced flag complex on $\Gamma$, and let $\varphi: G_\Gamma \to \ZZ$ be a rational character. Then
the kernel of $\varphi$ is of type $\mathbf{F_n}$ if and only if $\triangle_{\Gamma_L} < \triangle_\Gamma$ is $(n-1)$-connected and $(n-1)$-$\ZZ$-acyclic-dominating.   
\end{thm}

In this section, we give an algebraic proof of the fact that $\triangle_{\Gamma_L}$ is connected and $0$-acyclic-dominating if and only if the kernel is finitely generated and we exhibit an explicit (possibly infinite) presentation, see Theorem \ref{thm:finite_generation}. Furthermore, if  $\triangle_{\Gamma_L}$ is simply connected and $1$-acyclic-dominating, then we show that the kernel is finitely presented and we exhibit an explicit presentation, see Theorem \ref{thm:finite_presentation}.


Note that $0$-acyclic-domination requires that each vertex in $\triangle_\Gamma$ that is not in $\triangle_{\Gamma_L}$ be adjacent to a vertex in $\triangle_{\Gamma_L}$. Similarly, $1$-acyclic-domination requires that for each edge $(a_i,a_j)\in \triangle_\Gamma$, there is a vertex commuting with $a_i$ and $a_j$ that belongs to $\triangle_{\Gamma_L}$, i.e. either $a_i$ or $a_j$ are vertices in $\triangle_{\Gamma_L}$ or there exists $a_k$ such that $(a_i, a_k), (a_j,a_k)$  are edges in $\triangle_\Gamma$ and $a_k$ is a vertex in $\triangle_{\Gamma_L}$.

\textbf{Setting.} Any rational character from a RAAG $G_\Gamma$ can be viewed as an epimorphism $\varphi \colon G_\Gamma \twoheadrightarrow \ZZ$ sending each generator $a_i$ of $G_\Gamma$ to an integer ${n}_i\in \ZZ$. Since the map sending $a_i \to a_i^{-1}$ induces an automorphism of $G_\Gamma$, without loss of generality we may assume that $n_i \in \mathbb N \cup \{0\}$. Since $\varphi$ is surjective, there exists $w \in G_\Gamma$ such that $\varphi(w)=1$. Let $V(\Gamma)=\{a_1,\dots, a_s\}$ be the set of vertices of $\Gamma$. Without loss of generality (and up to relabelling the vertices), we can take  $w=a_{1}^{r_1} \cdots a_{\ell}^{r_\ell}$, where $n_i, r_i\ne 0$ for each $i=1, \dots, \ell$.

From the fact that $\varphi(w)=1$, we have that $n_{1}{r_1} + \cdots + n_{\ell}{r_\ell} =1$. Let $r=\max\{|r_1|, \dots, |r_\ell|\}$. 

Denote the kernel $\ker \varphi$ by $\gh$. Let $\Gamma_L <\Gamma$ be the induced subgraph of $\Gamma$ defined by the set of vertices $a_i$ such that $n_i\ne 0$ and let $\triangle_{\Gamma_L}$ be the corresponding induced flag complex in $\triangle_\Gamma$. Since by assumption $n_i\ne 0$ for $i=1, \dots, \ell$, we have that $a_i \in V(\Gamma_L)$ for $i=1, \dots, \ell$.

\textbf{The Reidemeister-Schreier presentation.}
The Reidemeister-Schreier method is a technique for producing presentations (in general infinite) of a subgroup $H$ of a group $G$ from the presentation of $G$. There are many variants of this method, the most common being the one suggested by Schreier, which chooses a set of transverse elements with the extra property that it is closed under subwords, to obtain a simpler presentation. In our case, we will use the variant described by Reidemeister, see for instance \cite[Theorem 2.8]{MKS}, which allows any choice of a transverse set at the price of some extra relations. We recall this version in the following

\begin{thm}[Reidemeister-Schreier, see Theorem 2.8 in \cite{MKS}] Let $G = \langle s \in S \mid r \in  R\rangle$ be a group. Let $H \unlhd G$ be a normal subgroup, let $T$ be a set of right coset representatives of $H$ in $G$, and let $\overline{\cdot} \colon G \to T$, $w \to \overline{w}$ be a right coset representative function. Then, $H$ has a presentation
$$
\langle x_{s,t} \mid x_{s,t}=\tau(ts (\overline{ts})^{-1}), \tau(t r t^{-1})=1,  s\in S, t\in T \rangle, 
$$
under the mapping $x_{s,t}\to ts (\overline{ts})^{-1}$ and where $\tau$ is the Reidemeister rewriting process for words in $H$, see {\rm Definition \ref{defn:Reidemeister_rewriting}}.
\end{thm}

\begin{defn}[Reidemeister rewriting process]\label{defn:Reidemeister_rewriting}
Let $H< G=\langle a_1, \dots, a_s\rangle$ be a subgroup of $G$ and let $w\to \overline{w}$ be a right coset representative function for $G$ mod $H$. Let $u\in H$ be written as a word
$$
u=a_{i_1}^{\epsilon_1} a_{i_2}^{\epsilon_2}\cdots a_{i_r}^{\epsilon_r},
$$
where $\epsilon_i = \pm 1$.

The rewriting process $\tau$ expresses the word $u$ as a product of the generators $\overline{w} a_i \overline{wa_i}^{-1}$ as follows:
$$
(\overline{w_1} a_{i_1}^{\epsilon_1} \overline{w_1a_{i_1}^{\epsilon_1}}^{-1}) (\overline{w_2} a_{i_2} \overline{w_2a_{i_2}^{\epsilon_2}}^{-1}) \dots (\overline{w_r} a_{i_r}^{\epsilon_r} \overline{w_r a_{i_r}^{\epsilon_r}}^{-1}),
$$
where
$$
w_1=1, w_2=w_1a_{i_1}^{\epsilon_1}, \dots, w_r=w_{r-1}a_{i_{r-1}}^{\epsilon_{r-1}}.
$$
\end{defn}

\textbf{Choice of right coset representative.} We next choose the transverse elements for the subgroup $\gh$ as follows. We define $w_m$ to be the following element $a_{1}^{mr_1} \cdots a_{\ell}^{mr_\ell}$ of $G_\Gamma$, where $m\in \ZZ$, $m\ne 0$; for $m=0$, we define $w_0=1$. From the definition, since $r_1n_1+\dots+r_\ell n_\ell=1$, we have that $\varphi(w_m)= m(r_1n_1+\dots+r_\ell n_\ell)=m$. Let us take $\{ w_m \, \, | \, \, m \in \ZZ\}$ as transverse set for the normal subgroup $\gh \leqslant G_\Gamma$.

From the Reidemeister--Schreier Theorem, we have that the kernel $\gh$ admits the following presentation:

\begin{thm}[Reidemeister--Schreier]\label{thm: R-S presentation}
In the above notation, 
 $$
 \gh = \langle x_{m,i} \, \, \mid \, \, x_{m,i}=\tau(w_m a_i (w_{m+n_i})^{-1}), \tau(w_m [a_i,a_j] w_{m}^{-1})=1 \rangle,
 $$
 where $x_{m,i}=w_m a_i (w_{m+n_i})^{-1}$, $m\in \ZZ$, $a_i\in V(\Gamma)$, and $(a_i, a_j) \in E(\Gamma)$. 
\end{thm}

\begin{rem}\label{rem:Reidemeister_process}
In our specific case, the Reidemeister rewriting process for the word $u=a_{i_1}^{\epsilon_1} a_{i_2}^{\epsilon_2}\cdots a_{i_r}^{\epsilon_r}$ in $H$ takes the form:
$$
\tau(u)=x_{k_1,i_1}^{\epsilon_1} x_{k_2, i_2}^{\epsilon_2} \dots x_{k_r,i_r}^{\epsilon_r},
$$
where $k_l$ is either $\varphi(a_{i_1}^{\epsilon_1}\dots a_{i_{l-1}}^{\epsilon_{l-1}})$ if $\epsilon_l=1$, or $\varphi(a_{i_1}^{\epsilon_1}\dots a_{i_{l}}^{\epsilon_{l}})$ if $\epsilon_l=-1$.

For instance, suppose that $a_{i_1}a_{i_2}^{-1}a_{i_3}\in \gh$, i.e. $n_{i_1}-n_{i_2}+n_{i_3}=0$. Then 
$$
\tau(a_{i_1}a_{i_2}^{-1}a_{i_3})=(a_{i_1} w_{n_{i_1}}^{-1}) (w_{n_{i_1}} a_{i_2}^{-1} w_{n_{i_1}-n_{i_2}}^{-1}) (w_{n_{i_1}-n_{i_2}} a_{i_3} w_{n_{i_1}-n_{i_2}+n_{i_3}}^{-1})= x_{0,i_1} \,x_{{n_{i_1}-n_{i_2}},i_2}^{-1} \, x_{{n_{i_1}-n_{i_2}},i_3}.
$$
\end{rem}

Note that the Reidemeister--Schreier presentation is independent of the geometry of $\triangle_\Gamma$.

\begin{lem} \label{lem:trasverse_homomorphism}
The map $\tau$ is independent of the choice of the word $\omega$ in the free group, i.e. if $\omega$ and $\omega'$ define the same element in the free group then $\tau(\omega)=\tau(\omega')$.
\end{lem}
\begin{proof}
Notice that $\tau(a_i a_i^{-1})= (a_i w_{n_i}^{-1})(w_{n_i} a_i^{-1} w_{n_i-n_i})=x_{0,i}x_{0,i}^{-1}$.

Suppose now that $\omega= \omega_1 a_i a_i^{-1} \omega_2$. Then 
$$
    \tau(\omega)= \tau(\omega_1 a_ia_i^{-1}\omega_2) = u_1 x_{\varphi(\omega_1),i} x_{\varphi(\omega_1),i}^{-1} u_2= u_1 u_2 = \tau(\omega_1\omega_2),
$$
where $u_1$ and $u_2$ are the words in the generators $x_{m,i}$ obtained in the Reidemeister rewriting process, see Definition \ref{defn:Reidemeister_rewriting} and Remark \ref{rem:Reidemeister_process}.
\end{proof}


\subsection{Notation} 
In our setting, assume that $\Gamma_L < \Gamma$ is connected.
\begin{itemize}

    \item (Paths) For any two vertices $a_i\neq a_j$ in $V(\Gamma_L)$, a path $q_{i,j}$ from $a_i$ to $a_j$ in $\Gamma_L$ is a sequence of vertices $a_i=a_{i_1}, a_{i_2}, \dots, a_{i_{d+1}}=a_j$ such that $a_{i_k}\in V(\Gamma_L)$ and $(a_{i_{k-1}}, a_{i_k})\in E(\Gamma_L)$, for $k=2, \dots, d+1$. The length of the path $d$ is the number of vertices minus $1$. If the path has length $0$, i.e. it is a vertex, say $a_i$, we denoted $q_{i,i}$ simply by $a_i$. If the path has length $1$, namely is given by two adjacent vertices $a_i, a_j$, we denote it by $e_{i,j}$. Sometimes we also think of paths of length more than $0$ as a sequence of edges $e_{i_1, i_2} e_{i_2, i_3} \dots e_{i_{d}, i_{d+1}}$.

    \item (Least common multiple of subsets of vertices) Denote by $N$ the least common multiple (lcm) of the $\{n_i\mid a_i\in V(\Gamma_L)\} \in \mathbb N$. For any $(a_i,a_j)\in E(\Gamma_L)$, let $N_{ij}$ be the lcm of $n_i$ and $n_j$. Let $p$ be a path in $\Gamma_L$, $p=a_{i_1}, \dots, a_{i_k}$. Then we define $N(p)$ to be the lcm of $n_{i_1}, \dots, n_{i_k}$. By definition, $N_{ij}, N(p)$ divide $N$ for any $(a_i, a_j)\in E(\Gamma_L)$ and any path $p$ in $\Gamma_L$.

    \item (Optimal paths) For the word $w$, we can define paths $q_w$ in $\Gamma_L$ as the concatenation of paths $q_{j_1,j_2}, q_{j_2,j_3}, \dots, q_{j_{t-1}, j_t}$, where $1=j_1<\dots< j_t=\ell$ (recall that we require that $n_i, r_i \ne 0$ for $i=1, \dots, \ell$ and so $a_1, \dots, a_\ell \in \Gamma_L$ and therefore $q_{j_r,j_s}\in \Gamma_L$).
    Note that, a priori, the paths $q_{i,k}$ above are not unique. We choose a path $q_w$ in $\Gamma_L$, such that the lcm $N(q_w)$ is minimal, denote it by $p_w$, and call such path \emph{optimal}. 

    Similarly, for any vertex $a_j \in V(\Gamma_L)$, we can consider paths $q_{j,w}$ that are the concatenation of a path $q_{j,1}$ and a path $q_w$ in $\Gamma_L$; among all such paths, we choose one so that $N(q_{j,w})=N_j\in \mathbb N$ is minimal. We denote such path $p_{j,w}$, call it \emph{optimal} and refer to the number $N_j$ as the \emph{local lcm} of the vertex $a_j$.

    \item (Weighted paths) For any integer $m$ which is divisible by $N_{ij}$, and $(a_i, a_j)\in E(\Gamma_L)$, we define the \textit{weighted edge} $e_{i,j}^m$ as $(a_i^{m/n_i} a_j^{-m/n_j})$. Similarly, for any integer $m$ divisible by $N(p)$, since any path $p=p_{i,j}$ in $\Gamma_L$ between two distinct vertices $a_i$ and $a_j$ (not necessarily adjacent) can be viewed as a sequence of consecutive edges in $\Gamma_L$, we define the \textit{weighted path} $p_{i,j}^m$ as the product of the weighted edges that define the path $p_{i,j}$, i.e. if $p_{i,j}=e_{i_1, i_2} e_{i_2, i_3} \dots e_{i_d, i_{d+1}}$, then $p_{i,j}^m=e_{i_1, i_2}^m e_{i_2, i_3}^m \dots e_{i_d, i_{d+1}}^m$. Notice that by definition, since $m$ is divisible by $N(p)$, it is also divisible by $N_{i_k, i_{k+1}}$ so the weighted edges are well-defined.

    Observe that in $G_\Gamma$, one has $e_{i,j}^m =(a_i^{m/n_i} a_j^{-m/n_j}) = (e_{j,i}^m)^{-1}$ and $p_{i,j}^m = (p_{j,i}^m)^{-1}$. On the other hand, while $e_{i,j}^m= e_{j,i}^{-m}$, we have that  $p_{i,j}^m \neq p_{j,i}^{-m}$ whenever $[a_i, a_j]\neq 1$. 
\end{itemize}

\begin{lem}\label{lem:edges}

    For each $a_j, a_k\in V(\Gamma_L)$ such that 
    $[a_j, a_k]=1$, we have that 
    $$e_{j,k}^{N_{jk}} \in \langle x_{m, i} \mid 0\leqslant m < N \rangle$$
     and, more precisely,
    $$
    e_{j,k}^{N_{jk}}= x_{0,j} x_{n_j,j}\cdots x_{N_{j,k}-n_j,j} x_{N_{j,k}-n_k, k}^{-1} \cdots x_{0,k}^{-1}.
    $$
    Therefore, $\langle e_{j,k}^{N_{j,k}} \mid (a_j, a_k)\in E(\Gamma_L)\rangle < \langle x_{m,i} \mid a_i \in V(\Gamma_L), 0 \leqslant m < N \rangle$.
\end{lem}    
\begin{proof}
Proof is by direct computation. Write    
$$
e_{j,k}^{N_{jk}}=(a_j^{N_{jk}/n_j} a_k^{-N_{jk}/n_k})= (a_j^{N_{jk}/n_j} w_{N_{jk}}^{-1})(a_k^{N_{jk}/n_k} w_{N_{jk}}^{-1})^{-1}.
$$
Now,
$$
(a_j^{N_{jk}/n_j} w_{N_{jk}}^{-1})= (a_j w_{n_j}^{-1}) (w_{n_j}a_j w_{2n_j}^{-1}) \cdots (w_{\frac{N_{jk}}{n_j}\cdot n_j-n_j}a_jw_{\frac{N_{jk}}{n_j}\cdot n_j}^{-1}).
$$ 
Note that $n_j < 2n_j < \cdots < \frac{N_{jk}}{n_j} \cdot n_j = N_{jk} \leqslant N$ and $0< n_j \leqslant N$ for $a_j\in V(\Gamma_L)$ and so $0\leqslant N_{jk}-n_j < N$. Hence, $(a_j^{N_{jk}/n_j} w_{N_{jk}}^{-1}) \in \langle x_{m,i} \mid a_i\in V(\Gamma_L), 0\leqslant m < N \rangle$ and the result follows.
\end{proof}

\begin{obs}\label{obs-1}
From Lemma \ref{lem:edges}, it follows that the weighted edges $e_{j,k}^m$, for $m$ a multiple of $N_{jk}$, also belong to the subgroup $\langle x_{m,i} \mid a_i \in V(\Gamma_L), 0\leqslant m < N  \rangle$.  

Indeed, if $m= tN_{jk}$, then $e_{j, k}^m = (a_j^{tN_{jk}/n_j} a_k^{-tN_{jk}/n_k})$ and since $[a_j,a_k]=1$, we have that $(a_j^{tN_{jk}/n_j} a_k^{-tN_{jk}/n_k})=(a_j^{N_{jk}/n_j} a_k^{-N_{jk}/n_k})^t$ and the later belongs to the subgroup since by Lemma~\ref{lem:edges}, $(a_j^{N_{jk}/n_j} a_k^{-N_{jk}/n_k}) \in \langle x_{m,i} \mid a_i \in V(\Gamma_L), 0 \leqslant m <N \rangle$. 
\end{obs}

We obtain the following

\begin{lem}\label{lem: better bound}
   Let $p=p_{j,k}$ be any path in $\Gamma_L$ and let $m$ be an integer divisible by $N(p)$. Then the weighted path $p_{j,k}^{m} \in \langle x_{m,i} \mid a_i \in V(\Gamma_L), 0\leqslant m< N \rangle$.
\end{lem}

\begin{proof}
   Proof follows from the definition of a weighted path, Lemma~\ref{lem:edges} and Observation~\ref{obs-1}.
\end{proof}

Recall that $r=\max\{|r_1|, \dots, |r_\ell|\}>0$ where $r_1n_1 + \dots + r_\ell n_\ell=1$.

\begin{thm}[Finite generation]\label{thm:finite_generation} The subgraph $\Gamma_L$ is connected and $0$-acyclic-dominating if and only if $\ker \varphi=\gh$ is finitely generated. More precisely,
$$
\gh = \langle x_{m,i} \, \, \mid \, \, x_{t,i}=\tau(w_t a_i w_{t+n_i}^{-1}), \tau(w_M [a_i,a_j] w_{M}^{-1})=1 \rangle,
$$
where $x_{k,i}=w_k a_i w_{k+n_i}^{-1}$, $0 \leqslant |m| < \ell rN^2$, $0\leqslant |t| < \ell^2r^2N^3$, $M\in \ZZ$, $a_i\in V(\Gamma)$, and $(a_i, a_j) \in E(\Gamma)$.
\end{thm}

\begin{rem}\label{rem:relations_in_generators}
Notice that in the family of relations $R_3$, the elements $x_{t,i}$ for $\ell rN^2  \leqslant |t| < \ell^2r^2N^3$ are not in the generating set and so formally, $x_{t,i}$ should be replaced by a word in the generators that represents it. We abuse the notation and keep $x_{t,i}$ in the relations for simplicity.  
\end{rem}

\begin{proof}
Assume that $\Gamma_L$ is connected and $0$-acyclic-dominating. From Theorem \ref{thm: R-S presentation}, we have that  $\gh = \langle x_{m,i}, m\in \ZZ, a_i\in V(\Gamma) \rangle$. Our goal is to prove that $\gh$ is finitely generated and admits a presentation as in the statement. 

We first show that $\gh$ is generated by $\langle x_{m,i} \mid 0\leqslant |m| < lrN^2, a_i \in V(\Gamma)\rangle$ and when expressing $\tau(w_M a_i w_{M+n_i}^{-1})$ for $|M| \geqslant l^2r^2N^3$ in terms of the finite set of generators, we will see that the generator $x_{M,i}$ does not appear in this expression and so it can be removed together with the relation $x_{M,i}=\tau(w_M a_i (w_{M+n_i})^{-1})$, using a Tietze move.

Recall that
\begin{equation}\label{equ: start equ}
   x_{m,i}=w_ma_iw_{m+n_i}^{-1} = a_{1}^{mr_1} \cdots a_{\ell}^{mr_\ell} a_i (a_{1}^{(m+n_i)r_1} \cdots a_{\ell}^{(m+n_i)r_\ell})^{-1}. 
\end{equation}

Divide $mr_j$ by $N$ with remainder, $mr_j=k_jN+s_j$; $j= 1, \ldots, \ell$, here $s_j < N$. Rewriting Equation~\eqref{equ: start equ} we have
$$
    \begin{array}{l}
    \vspace{0.1cm}
          w_ma_i(w_{m+n_i})^{-1} \\
         \vspace{0.1cm}
         
         = a_{1}^{mr_1} \, \, a_{2}^{mr_2} \cdots a_{\ell}^{mr_\ell} \, \, a_i  \, \,(w_{m+n_i})^{-1}\\
         \vspace{0.1cm}
         
         = a_{1}^{s_1} \, \, a_{1}^{k_1N} \cdots a_{\ell}^{k_\ell N} \, \, a_{\ell}^{s_\ell} \, \, a_i \, \, (w_{m+n_i})^{-1}\\
         \vspace{0.1cm}
         
         = a_{1}^{s_1} \{a_{1}^{k_1N}a_{2}^{-(Nk_1n_1)/n_2}\} \, \, a_{2}^{s_2} \{a_{2}^{N(k_1n_1+ k_2n_2)/n_2} a_3^{-N(k_1n_1+ k_2n_2)/n_3}\} \cdots \\
         \quad \quad \quad \quad \cdots \{a_{\ell}^{Nd/n_\ell} a_j^{-Nd/n_j}\} \, \, a_i \, \, a_j^{Nd/n_j} (w_{m+n_i})^{-1}\\
         \vspace{0.1cm}

         = a_{1}^{s_1} \, \, p_{1, 2}^{Nk_1n_1} \, \, a_{2}^{s_2} \, \, p_{2, 3}^{N(k_1n_1+k_2n_2)} \cdots a_{\ell}^{s_\ell} \, \, p_{\ell, j}^ {Nd} \, \, a_i \, \, a_j^{Nd/n_j} (w_{m+n_i})^{-1},
\end{array}
$$
where $d= k_1n_1 + \cdots + k_\ell n_\ell$, $p_{k,j}$ are paths in $\Gamma_L$, and $a_j \in V(\Gamma_L)$ is adjacent to $a_i$ in $\Gamma$. Notice that the paths $p_{k,k+1}$, $k=1, \dots, \ell-1$, exist since $\Gamma_L$ is connected and $p_{\ell,j}$ exists (with $a_j\in V(\Gamma_L)$ and either $a_i=a_j$ or $(a_i,a_j)\in E(\Gamma)$) since $\triangle_{\Gamma_L}$ is $0$-acyclic-dominating. Let, $d' = n_1s_1 + \cdots + n_\ell s_\ell$. Continuing the same procedure, we arrive at the following:
\begin{equation} \label{eq:maineq}
\begin{array}{l}  
\vspace{0.1cm}
      w_ma_i(w_{m+n_i})^{-1}\\
      \vspace{0.1cm}
      
        = a_{1}^{s_1} \, \, p_{1, 2}^{Nk_1n_1} \, \, a_{2}^{s_2} \, \, p_{2, 3}^{N(k_1n_1+k_2n_2)} \cdots a_{\ell}^{s_\ell} \, \, p_{\ell, j}^{Nd} \, \, a_i \, \, p_{j, \ell}^{Nd}  \, \, a_{\ell}^{-s_\ell-r_\ell n_i} \cdots p_{2, 1}^{Nk_1n_1} \, \, a_{1}^{-s_1-r_1n_i}\\
       \vspace{0.1cm}

       = a_{1}^{s_1} \, \, p_{1, 2}^{Nk_1n_1} \, \, a_{2}^{s_2} \, \, p_{2, 3}^{N(k_1n_1+k_2n_2)} \cdots a_{\ell}^{s_\ell} \, \, p_{\ell, j}^{Nd} \, \, a_i \, \, (p_{\ell, j}^{Nd})^{-1} \, \, a_{\ell}^{-s_\ell-r_\ell n_i} \cdots (p_{1, 2}^{Nk_1n_1})^{-1} \, \, a_{1}^{-s_1-r_1n_i}\\
       \vspace{0.1cm}
        
        = (a_{1}^{s_1} \, \, w_{n_1s_1}^{-1}) \, \,(w_{n_1s_1} \, \, p_{1, 2}^{Nk_1n_1} \, \, w_{n_1s_1}^{-1})
        (w_{n_1s_1}a_{2}^{s_2} w_{n_1s_1+n_2s_2}^{-1})
        \cdots (w_{d'-n_\ell s_\ell}a_{\ell}^{s_\ell} w_{d'}^{-1}) \\        
        (w_{d'} \, \, a_i \, \, w_{n_i+d'}^{-1}) 
        (w_{n_i+d'}(p_{\ell, j}^{Nd})^{-1} w_{n_i+d'}^{-1})(w_{n_i+d'} a_{\ell}^{-s_\ell-r_\ell n_i} w_{n_i+d'-n_\ell s_\ell -  n_\ell n_i  r_\ell}^{-1})\\
        \cdots (w_{n_in_1r_1+n_1s_1} \, \, a_{1}^{-s_1-r_1n_i} \, \, w_{(n_in_1r_1+n_1s_1) -n_1s_1-n_in_1r_1}^{-1}).
        \end{array}
\end{equation}

Every term in the last expression of Equation \eqref{eq:maineq} has one of the two following forms:
    \begin{enumerate}
        \item\label{it:1} $w_xa_i^yw_{x+n_iy}^{-1}$,
        \item\label{it:2} $w_x p_{j,k}^{t} w_x^{-1}$, $w_{x+n_iy} p_{j,k}^{t} w_{x+n_iy}^{-1}$, $p_{j,k}^{t}$ is a weighted path in $\Gamma_L$ joining  two vertices.
    \end{enumerate}
Furthermore, using the definition of $d'$, we get that 
$$
|x|,|x+n_iy| \leqslant \max_{1\leqslant k \leqslant \ell} \{ |n_1s_1|+ \dots +|n_k s_k| + n_i |( 1 - \sum_{k\leqslant j \leqslant \ell} n_jr_j)|\}. 
$$
Since $1 = \sum_{1\leqslant j \leqslant \ell} n_jr_j$, we have that $1 - \sum_{k\leqslant j \leqslant \ell} n_jr_j= \sum_{1\leqslant j <k} n_jr_j$ and so 
$$
|x|,|x+n_iy| \leqslant \max_{1\leqslant k \leqslant \ell} \{ |n_1s_1|+ \dots +|n_k s_k| + n_i |\sum_{1\leqslant j <k} n_jr_j|\}. 
$$
Since $|n_is_i|<N^2$, $|n_in_jr_j|< rN^2$ for each $1\leqslant i,j\leqslant \ell$ and $N^2 < rN^2$, we conclude that
$$
|x|, |x+n_iy| < \ell r N^2.
$$

Write each term $w_xa_i^y(w_{x+n_iy})^{-1}$ as follows
\begin{equation}\label{eq:first_type}
w_x a_i^yw_{x+n_iy}^{-1}= (w_x a_i w_{x+n_i}^{-1})(w_{x+n_i} a_i(w_{x+2n_i})^{-1})\cdots (w_{x+(y-1)n_i} a_i(w_{x+n_iy})^{-1}).
\end{equation}

\paragraph{Type (1) terms} Terms of the form $w_xa_i^yw_{x+n_iy}^{-1}$ in Equation \eqref{eq:maineq} can be written as a product of generators from the theorem as in Equation \eqref{eq:first_type}, since in that case, $|x|$ and $|x+n_iy|$ are bounded by $\ell r N^2$. 

\paragraph{Type (2) terms} In this case, we write $p_{j,k}^{t}$ as a product of weighted edges, $e_{j_1, j_2}^{t} = (a_{j_1}^{t/n_{j_1}} a_{j_2}^{-t/n_{j_2}})$ and we write $w_x p_{j,k}^{t} w_x^{-1}$ as a product of conjugates of weighted edges $w_x e_{j_1, j_2}^{t} w_x^{-1}$.



Recall that $t$ is a multiple of $N$, say $t=qN$. Since $[a_{j_1},a_{j_2}]=1$, we can write $w_x e_{j_1, j_2}^{t} w_x^{-1}$ as a product of $q$ factors $w_x e_{j_1, j_2}^{N} w_x^{-1}$. 

We are left to show that $w_x e_{j_1, j_2}^{N} w_x^{-1}$ can be written as a product of the generators from the statement of the theorem. 

Write
$$
w_x e_{j_1, j_2}^{N} w_x^{-1}=(w_x a_{j_1}^{N/n_{j_1}} w_{x+N}^{-1}) (w_{x+N} a_{j_2}^{-N/n_{j_2}} w_x^{-1}).
$$

Since $|x|, |x+n_iy| < \ell r N^2$, type (2) terms $w_x p_{j,k}^{t} w_x^{-1}$ and $w_{x+n_iy} p_{j,k} w_{x+n_iy}$ in Equation \eqref{eq:maineq} can be written as a product of the generators $x_{m,i}$ for $0\leqslant |m| < \ell rN^2$ and $a_i\in V(\Gamma)$ (in the free group $F(a_i\mid a_i\in V(\Gamma))$).

From Lemma \ref{lem:trasverse_homomorphism}, we have that that $\tau(w_M a_i w_{M+n_i}^{-1})$ is the same as the product of some $\tau(w_m a_i w_{m+n_i}^{-1})$ where $|m| < \ell rN^2$. 

Furthermore, 
\begin{equation} \label{eq:maineq 2}
\begin{array}{l}  
\vspace{0.1cm}
      \tau(w_ma_i(w_{m+n_i})^{-1})\\
      \vspace{0.1cm}
      
        = \tau(a_1^{r_1m} \cdots a_\ell^{r_\ell m} a_i a_\ell^{-r_l(m+n_i)} \cdots a_1^{-r_1(m+n_i)})\\
        \vspace{0.1cm}

        = (a_1 w_{n_1}^{-1}) \cdots (w_{r_1n_1m - n_1} a_1 w_{r_1n_1m}^{-1})\\
        (w_{r_1n_1m} a_2 w_{r_1n_1m +n_2}^{-1}) \cdots (w_{r_1n_1m+r_2n_2m -n_2}a_2w_{r_1n_1m+r_2n_2m}^{-1}) \cdots \\
        \cdots (w_{r_1n_1m+ \dots +r_{\ell-1}n_{\ell-1}m}a_\ell w_{r_1n_1m+ \dots +r_{\ell-1}n_{\ell-1}m+n_\ell}^{-1}) \cdots (w_{r_1n_1m+\dots+r_\ell n_\ell m-n_\ell}a_\ell w_{r_1n_1m+\dots+r_\ell n_\ell m}^{-1})
        \\
        (w_m a_i w_{m+n_i}^{-1}) (w_{m+n_i}a_{\ell}^{-1}w_{m+n_i-n_\ell}^{-1}) \cdots (w_{n_1}a_1^{-1} w_0) \\
        \vspace{0.1cm}

        = x_{0,1} \cdots x_{r_1n_1m-n_1,1}\cdots 
        x_{m-n_\ell, \ell} x_{m,i} \cdots x_{0,1}^{-1}.
        \end{array}
\end{equation}

Notice that for all the generators $x_{t,i}$ in the above equation, we have that $|t|\leqslant \ell rNm$ and since $|m|\leqslant \ell rN^2$, we conclude that $|t|\leqslant \ell^2r^2N^3$.

It follows that the generators $x_{M,i}$ do not appear in the expression of $\tau(w_M a_i w_{M+n_i}^{-1})$ for $|M| \geqslant \ell^2r^2N^3$, therefore, we can apply a Tietze transformation and remove the generator and the relation and the result follows.

Assume now that $\triangle_{\Gamma_L}$ is either not connected or not $0$-acyclic dominating. Observe that, for any induced subgraph $\Gamma_L \leqslant \Upsilon \leqslant \Gamma$ containing $\Gamma_L$, we have that $G_\Gamma$ retracts to $G_\Upsilon$ and the epimorphism $\varphi:G_\Gamma \to \ZZ$ factors through the epimorphism $\varphi|_{G_\Upsilon}: G_\Upsilon \to \ZZ$ given by the restriction. If $\Gamma_L$ is not connected, then $G_{\Gamma_L}$ is a nontrivial free product, and if it is not $0$-acyclic-dominating and $a_l$ is a vertex in $\Gamma$ not connected to $\Gamma_L$, taking $\Upsilon$ to be the induced subgraph defined by $\Gamma_L$ and $a_l$, we have that $G_\Upsilon$ is also a nontrivial free product. Since by a result by Baumslag, see \cite{Baumslag}, nontrivial finitely generated normal subgroups of free products are of finite index, we have that the kernel $K_\Upsilon$ of the epimorphism $\varphi: G_\Upsilon \to \ZZ$ is not finitely generated and since the epimorphism $G_\Gamma \to G_\Upsilon$ induces and epimorphism of kernels $\gh \twoheadrightarrow K_\Upsilon$, we conclude that $\gh$ is not finitely generated. 
\end{proof}

Notice that in the case when $n_i=1$ for all $i$, that is, when the kernel is a Bestvina-Brady group, then we have that $\Gamma=\Gamma_L$ and so $\Gamma_L$ being connected and $0$-acyclic-dominating is equivalent to $\Gamma$ being connected. In this case, we also have that $l=r=N=1$ and we can take $w=a_1$. Therefore, we recover the well-known fact:

\begin{cor}\label{cor:fg_BB}
Let $\Gamma$ be connected and let $H_\Gamma$ be the Bestvina-Brady subgroup of $G_\Gamma$. Then $H_\Gamma$ is finitely generated and admits an {\rm(}infinite{\rm)} presentation
$$
\gh = \langle x_{0,i} \, \, \mid \, \, \tau(w_M [a_i,a_j] w_{M}^{-1})=1 \rangle,
$$
where $x_{0,i}=a_ia_1^{-1}$, $i=1, \dots, s$, $M\in \ZZ$, $(a_i,a_j)\in E(\Gamma)$.  
\end{cor}

\begin{thm}[Presentation] \label{thm:presentation}
Let $\Gamma_L$ be connected and $1$-acyclic-dominating. Then $\ker \varphi=\gh$ has an {\rm(}infinite{\rm)} presentation of the form:
    $$
    \langle x_{m,i} \, \, \mid \, \, R_1, R_2, R_3\rangle, 
    $$
    where $x_{m,i}=w_m a_i w_{m +n_i}^{- 1}$, $0 \leqslant |m| < \ell rN^2$, and the relations $R_1, R_2$ and $R_3$ are defined as follows:  
\begin{itemize}
 \item[$R_1$:] for any directed cycle $(a_{i_1}, a_{i_2}, \ldots, a_{i_k}=a_{i_1})$ with $k \geqslant 3$,
\[
 \tau(w_s
 e_{i_1,i_2}^{Nq} e_{i_2,i_3}^{Nq} \cdots e_{i_{k-1},i_k}^{Nq}  w_s^{-1})=1, \hspace{0.2cm} 0 \leqslant |s| < N, q\in \ZZ;
\]
\item[$R_2$:] $\tau(w_s [a_i, a_j] w_s^{-1})=1, \hspace{0.2cm} 0 \leqslant |s| < N$, where $(a_i, a_j) \in E(\Gamma)$.
\item [$R_3$:] $x_{t,i}=\tau(w_t a_i (w_{t+n_i})^{-1})$ for $0\leqslant |t| < \ell^2r^2N^3$.
\end{itemize}
\end{thm}
\begin{proof}
Since $\Gamma_L$ is connected and $1$-acyclic-dominating, from Theorem \ref{thm:finite_generation} we have that the kernel is finitely generated by $w_m a_i^{\pm 1} w_{m \pm n_i}^{- 1}$, where $0 \leqslant |m| < \ell^2 r^2N^3$.

Recall that by Observation \ref{obs-1}, all relations from $R_1$ can be written as words in the generators of $\gh$.

We are left to prove that $\Gamma_L$ admits a presentation with relations $R_1, R_2, R_3$. We first show that the normal subgroup generated by the Reidemeister--Schreier relations $\{w_M[a_i,a_j] w_M^{-1}\}$ given in Theorem \ref{thm: R-S presentation} is contained in the normal closure of the relations $R_1$ and $R_2$ in the free group $F(a_i\mid a_i\in V(\Gamma))$.


     Let us first consider the case when $M$ is a multiple of $N$: $w_{kN}[a_i, a_j]w_{kN}^{-1}$. We recall that, $N$ is the lcm of $n_i$'s and  $w_{kN}=a_{1}^{kNr_1} \cdots a_{\ell}^{kNr_\ell}$. Thus,
$$
\begin{array}{l}
    \vspace{0.1cm}
       w_{kN}[a_i, a_j]w_{kN}^{-1}  \\
       \vspace{0.1cm}
       
        =a_{1}^{kNr_1} \cdots a_{\ell}^{kNr_\ell} \, \, [a_i, a_j] \, \, a_{\ell}^{-kNr_\ell} \cdots a_{1}^{-kNr_1}\\
        \vspace{0.1cm}

        = (a_{1}^{kNr_1} \, \, a_{2}^{-kNr_1n_1/n_2}) \cdots
        (a_{{\ell -1}}^ {{kN(r_1n_1+ \cdots r_{\ell -1} n_{\ell -1})}/n_{\ell-1}} \, \, a_{\ell} ^{-{kN(r_1n_1+ \cdots r_{\ell -1} n_{\ell -1})}/n_{\ell}})\\
        
        (a_{\ell}^{{kN(r_1n_1+ \cdots r_\ell n_\ell)}/n_\ell} \, \, a_d^{-{kN(r_1n_1+ \cdots r_\ell n_\ell)}/n_d}) \, \, [a_i, a_j] \, \, a_d^{kN/n_d} \, \, w_{kN}^{-1}.
        
    \end{array}
$$

Recall that $r_1n_1 + \cdots + r_\ell n_\ell =1$. Using the definition of weighted path, we have that:
$$
 w_{kN}[a_i, a_j]w_{kN}^{-1}           
        = p_{1,2}^{kNr_1n_1} \, \, p_{2,3}^{kN(r_1n_1+r_2n_2)} \cdots p_{{\ell -1},\ell}^{kN(r_1n_1+r_2n_2+ \cdots +r_{\ell-1} n_{\ell-1})} \, \, p_{\ell, d}^{kN} \, \, [a_i, a_j] \, \, p_{d,1}^{kN} \, \, p_{1, \ell}^{kN} \, \, X^{-1},
$$
where $X= p_{1,2}^{kNr_1n_1} \, \, p_{2,3}^{kN(r_1n_1+r_2n_2)} \cdots p_{{\ell -1},\ell}^{kN(r_1n_1+r_2n_2+ \cdots +r_{\ell-1} n_{\ell-1})}$, the paths, $p_{r,s}$ are in $\Gamma_L$, $a_d\in V(\Gamma_L)$ and $(a_d,a_i), (a_d,a_j)\in E(\Gamma)$. The paths $p_{r,s}$ and the vertex $a_d$ with the required properties exist since $\triangle_{\Gamma_L}$ is $1$-acyclic-dominating.
Hence,
\[
w_{kN}[a_i, a_j]w_{kN}^{-1}= X \, \, p_{\ell, d}^{kN} \, \, [a_i, a_j]  \, \, p_{d,1}^{kN} \, \, p_{1, \ell}^{kN} \, \, X^{-1}.
\]

Then, conjugating by $(p_{\ell, d}^{kN})^{-1}X^{-1}$, we see that $w_{kN}[a_i, a_j]w_{kN}^{-1}$ is conjugate in $\gh$ to $[a_i, a_j] C_i^k$, where $C_i^k= \, p_{d,1}^{kN} \, \, p_{1, \ell}^{kN}p_{\ell, d}^{kN}$. Since  $[a_i,a_j] C_d^k \in \langle R_1, R_2\rangle$, we conclude that $w_{kN}[a_i, a_j]w_{kN}^{-1}\in \langle\langle R_1, R_2\rangle\rangle$.


Now we consider the general case, i.e. $w_M [a_i, a_j] w_M^{-1}$ for an arbitrary integer $M$. Dividing $M$ by $N$ with remainder, we write $M=kN+s$ for $s < N$.

We have the following:
$$
\begin{array}{l}
\vspace{0.1cm}
   w_{kN+s} [a_i, a_j]  (w_{kN+s})^{-1}  \\
   \vspace{0.1cm}
   
   = a_{1}^{r_1s} \, \, a_{1}^{r_1kN} \, \, a_{2}^{r_2s} \, \, a_{2}^{r_2kN} \cdots a_{\ell}^{r_\ell s} \, \, a_{\ell}^{r_\ell kN} \, \, [a_i, a_j] \, \, (w_{kN+s})^{-1} \\
   \vspace{0.1cm}
   
   = a_{1}^{r_1s} \, \, a_{2}^{r_2s} \cdots a_{\ell}^{r_\ell s} \, \, [a_{\ell}^{-r_\ell s} \cdots a_{2}^{-r_2 s} \, \, (a_{1}^{r_1kN} \, \, a_{2}^{-r_1n_1kN/n_2}) \, \, a_{2}^{r_2s} \cdots a_{\ell}^{r_\ell s}] \\
   \vspace{0.1cm}
   
   [a_{\ell}^{-r_\ell s} \cdots a_{3}^{-r_3 s} \, \, (a_{2}^{(r_1n_1+r_2n_2)kN/n_2} \, \, a_{3}^{-(r_1n_1+r_2n_2)kN/n_3}) \, \, a_{3}^{r_3s} \cdots a_{\ell}^{r_\ell s}]
   \cdots \\
   \vspace{0.1cm}
   
   \cdots [a_{\ell}^{-r_\ell s} \, \, (a_{{\ell -1}}^{(r_1n_1+r_2n_2+ \cdots r_{\ell -1}n_{\ell -1})kN/n_{\ell -1}} \, \, a_{\ell}^{-(r_1n_1+r_2n_2+ \cdots r_{\ell -1}n_{\ell -1})kN/n_\ell}) \, \, a_{\ell}^{r_\ell s}] \, \, (a_{\ell}^{kN/n_{\ell}} \, \, a_{d}^{-kN/n_{d}}) \\
   
   [a_i, a_j ] a_d^{kN/n_d} (w_{kN+s})^{-1} 
   \\
   \vspace{0.1cm}

   = w_s \, \, [a_{\ell}^{-r_\ell s} \cdots a_{2}^{-r_2 s} \, \, p_{1, 2}^{r_1n_1kN} \, \, a_{2}^{r_2s} \cdots a_{\ell}^{r_\ell s}] \cdots [a_{\ell}^{-r_\ell s} \, \, p_{{\ell -1}, \ell}^{(r_1n_1+r_2n_2+ \cdots r_{\ell -1}n_{\ell -1})kN} \, \, a_{\ell}^{r_\ell s}] \, \, p_{\ell, d}^{kN} \\

   [a_i, a_j ] \, \,  a_d^{kN/n_d} \, \, (w_{kN+s})^{-1},
\end{array}
$$
where $p_{u,v}$ are paths in $\Gamma_L$, $a_d\in V(\Gamma_L)$ and $(a_d,a_i), (a_d, a_j)\in E(\Gamma)$.

Let $X = [a_{\ell}^{-r_\ell s} \cdots a_{2}^{-r_2 s} \, \, p_{1, 2}^{r_1n_1kN} \, \, a_{2}^{r_2s} \cdots a_{\ell}^{r_\ell s}] \cdots [a_{\ell}^{-r_\ell s} \, \, p_{{\ell -1}, \ell}^{(r_1n_1+r_2n_2+ \cdots r_{\ell -1}n_{\ell -1})kN} \, \, a_{\ell}^{r_\ell s}]$.

Then, 
\begin{equation}\label{equ: rel_gen}
 w_{kN+s} \, \, [a_i, a_j] \, \, (w_{kN+s})^{-1} =  w_s \, \, X \, \, p_{\ell, d}^{kN} \, \, [a_i, a_j ] \, \, p_{d, 1} ^{kN} \, \, p_{1, \ell}^{kN} \, \, X^{-1} \, \, w_s^{-1}.
\end{equation}

Conjugating \eqref{equ: rel_gen} by $w_s X^{-1} w_s^{-1}$, we obtain 
\begin{gather}\notag
\begin{split}
w_{kN+s} \, \, [a_i, a_j] \, \, (w_{kN+s})^{-1} & \sim w_s \, \, p_{\ell, d}^{kN} \, \, [a_i, a_j ] \, \, p_{d, 1} ^{kN} \, \, p_{1, \ell}^{kN} \, \, w_s^{-1} \\
&= (w_s \, \, p_{\ell, i}^{kN} \, \, w_s^{-1}) \, \, w_s \, \, [a_i, a_j ] \, \, p_{d, 1} ^{kN} \, \, p_{1, \ell}^{kN} \, \, w_s^{-1}
\end{split}
\end{gather}
in $\gh$ (note that $w_s X^{-1} w_s^{-1} \in \gh$). Since $w_s \, \, (p_{ \ell, d}^{kN})^{-1} \, \, w_s^{-1} \in \gh$ again conjugating by $w_s \, \, (p_{\ell, d}^{kN})^{-1} \, \, w_s^{-1}$, we see that 
$$
w_{kN+s} \, \, [a_i, a_j] \, \, (w_{kN+s})^{-1} \sim w_s \, \, [a_i, a_j]  \, \, p_{d,1}^{kN} \, \, p_{1,\ell}^{kN} \, \, p_{\ell, d}^{kN} \, \, w_s^{-1} =( w_s [a_i, a_j] w_s^{-1}) \, \,( w_s C_d^{k} w_s^{-1}),
$$
where $C_d^k= \, p_{d,1}^{kN} \, \, p_{1, \ell}^{kN}p_{\ell, d}^{kN}$.  Therefore, $w_{M}[a_i, a_j]w_{M}^{-1} \in \langle\langle R_1, R_2 \rangle\rangle$.

Straightforward computation shows that cycles define the trivial element in $\gh$ and so $R_1,R_2, R_3 \in \langle\langle R_3, w_{M}[a_i, a_j]w_{M}^{-1}, M\in \ZZ\rangle\rangle$. Now the result follows from Lemma \ref{lem:trasverse_homomorphism}.
\end{proof}

Notice that in the case when $n_i=1$ for all $i$, that is, when the kernel is a Bestvina-Brady group, then we have that $\Gamma=\Gamma_L$ and so $\Gamma_L$ being connected and $1$-acyclic-dominating is equivalent to $\Gamma$ being connected. In this case, we also have that $\ell=r=N=1$ and we can take $w=a_1$. Using Corollary \ref{cor:fg_BB}, we recover Dicks and Leary's presentation:

\begin{cor}[cf. Theorem 1 of \cite{DL}]
Let $\Gamma$ be connected and let $\gh$ be the Bestvina-Brady subgroup of $G_\Gamma$. Then $H_\Gamma$ admits the presentation
$$
    \langle x_{0,i} \, \, \mid \, \, \tau(e_{i_1,i_2}^{q} e_{i_2,i_3}^{q} \cdots e_{i_{k-1},i_k}^{q}) \rangle, 
    $$
    where $x_{0,i}= a_ia_1^{- 1}$, $a_i\in V(\Gamma)$, $e_{i_r,i_s}^m = a_{i_r}^{m} a_{i_s}^{-m}$, $(a_{i_r}, a_{i_s})\in E(\Gamma_L)$,$(e_{i_1, i_2}, \ldots, e_{i_{k-1}, i_1})$ is a directed cycle with $k \geqslant 3$, and $ q \in \ZZ$.
\end{cor}

\begin{thm}[Finite presentation]\label{thm:finite_presentation}
Let $\triangle_{\Gamma_L}$ be simply connected and $1$-acyclic-dominating. Then the kernel of $\ker \varphi=\gh$ has the following finite presentation 
    $$
    \langle x_{m,i} \, \, \mid \, \, R_1', R_2, R_3\rangle, 
    $$
    where $x_{t,i}=w_t a_i w_{t+n_i}^{-1}$, $0 \leqslant |m| < \ell r N^2$. The relations $R_1', R_2, R_3$ are defined as follows:  
\begin{itemize}
\item[$R_1'$:] for any directed $3$-cycle $(a_{i_1}, a_{i_2}, a_{i_3})$,
\[
 \tau(w_s
 e_{i_1,i_2}^{Nq} e_{i_2,i_3}^{Nq} e_{i_{3},i_1}^{Nq}  w_s^{-1}), \hspace{0.2cm} 0 \leqslant |s| < N, q=\pm 1;
\]

\item[$R_2$:] $\tau(w_s [a_i, a_j] w_s^{-1})=1, \hspace{0.2cm} 0 \leqslant |s| < N$, where $(a_i, a_j) \in E(\Gamma)$.
\item [$R_3$:] $x_{t,i}=\tau(w_t a_i (w_{t+n_i})^{-1})$ for $0\leqslant |t| < \ell^2r^2N^3$.
\end{itemize}
\end{thm}

\begin{proof}
The only infinite family of relations in the presentation given in Theorem \ref{thm:presentation} is the family $R_1$. The proof that this set is finite when $\triangle_\Gamma$ is simply connected is analogous to that of \cite[Proposition 2]{DL}. In short, if $\triangle_\Gamma$ is simply connected, cycles of length $3$ in $\Gamma$ generate the fundamental group of the graph $\Gamma$ and, in this case, the authors prove that the relations for $n$-cycles are a consequence of the relations for $3$-cycles (together with $R_2, R_3$). In addition, if the edges $e_1e_2e_3$ define an oriented triangle, from the values $q=1,-1$ in the relations, one obtains that $[e_1,e_2]=1$ and $e_1e_2=e_3$. From this commutation relation, one deduces that the relations for other values of $q\in \ZZ$ are consequences. See \cite[Proposition 2]{DL} for details. 
\end{proof}

Notice that in the case when $n_i=1$ for all $i$, that is when the kernel is a Bestvina-Brady group, we have that $\triangle_\Gamma=\triangle_{\Gamma_L}$ and so $\triangle_{\Gamma_L}$ being simply connected and $1$-acyclic-dominating is equivalent to $\triangle_\Gamma$ being simply connected. In this case, we also have that $\ell=r=N=1$ and, without the loss of generality, we can take $w=a_1$. We recover Dicks and Leary's presentation:

\begin{cor}[cf. Proposition 2 and Corollary 3 of \cite{DL}]
Let $\triangle_\Gamma$ be simply connected and let $H_\Gamma$ be the Bestvina-Brady subgroup of $G_\Gamma$. Then $H_\Gamma$ admits the finite presentation
$$
    \langle x_{0,i} \, \, \mid \, \, \tau(e_{i_1,i_2}^{q} e_{i_2,i_3}^{q} e_{i_3,i_1}^{q}) \rangle, 
    $$
    where $x_{0,i}= a_ia_1^{- 1}$, $a_i\in V(\Gamma)$, $e_{i_r,i_s}=a_{i_r}a_{i_s}^{-1}$ and $e_{i_1,i_2} e_{i_2,i_3} e_{i_3,i_1}$ is a directed $3$-cycle, and $q=\{\pm 1\}$.
\end{cor}

\end{document}